\documentclass{ws-ijnt}

\usepackage{subcaption}
\usepackage{enumitem}
\usepackage{mathtools}
\usepackage{hyperref}
\usepackage[babel]{csquotes} 
\usepackage{tikz}
\usepackage{array}
\usepackage{cite}
\usepackage{tikz-cd}
\usepackage{faktor } 
\usetikzlibrary{arrows,patterns}
\usetikzlibrary{matrix,arrows,shapes,shapes.misc}
\usetikzlibrary{decorations.pathreplacing,angles,quotes}
\usepackage{relsize}

\tikzset{cross/.style={cross out, draw=black, minimum size=2*(#1-\pgflinewidth), inner sep=0pt, outer sep=0pt},	
	cross/.default={1pt}}

\newtheorem{fact}[theorem]{Fact}

\renewcommand\epsilon{\varepsilon}

\newcommand\F{\mathbb{F}}
\newcommand\Fq{\mathbb{F}_q}
\newcommand\Z{\mathbb{Z}}

\newcommand\R{\mathbb{R}}
\newcommand\M{\mathbb{M}}
\newcommand\PP{\mathbb{P}}
\newcommand\K{\mathbb{K}}
\newcommand\N{\mathbb{N}}
\newcommand\T{\mathbb{T}}
\newcommand\A{\mathbb{A}}
\newcommand\bft{\mathbf{t}}
\newcommand\bfun{\mathbf{1}}
\newcommand\bfs{\mathbf{s}}
\newcommand\bfx{\mathbf{x}}
\newcommand\bfX{\mathbf{X}}

\newcommand\bfy{\mathbf{y}}
\newcommand\surj{\text{surj}}

\newcommand\st{\text{ such that }}
\newcommand\AND{\text{ and }}
\newcommand\Magma{{\sc Magma}}
\newcommand\PC{\mathsf{PC}}
\newcommand\C{\mathsf{C}}

\DeclareMathOperator{\Pic}{Pic}
\DeclareMathOperator{\Red}{Red}

\DeclareMathOperator{\Cone}{Cone}
\DeclareMathOperator{\ev}{ev}
\DeclareMathOperator{\Span}{Span}
\DeclareMathOperator{\GL}{GL}

\DeclareMathOperator{\LM}{LM}
\DeclareMathOperator{\Conv}{Conv}

\newcolumntype{C}[1]{>{\centering\let\newline\\\arraybackslash\hspace{0pt}}m{#1}}

\newcommand\ps[2]{{\left\langle#1,#2\right\rangle}}

\title{Projective toric codes}

\author{Jade Nardi}

\address{INRIA \\ LIX CNRS UMR 7161 \\
	Ecole Polytechnique, 91128 Palaiseau Cedex, France \\
	\email{jade.nardi@inria.fr}}

\date{\today}

\begin{document}

	\maketitle
	
	\begin{abstract}
		Any integral convex polytope $P$ in $\R^N$ provides a $N$-dimensional toric variety $\bfX_P$ and an ample divisor $D_P$ on this variety. This paper gives an explicit construction of the algebraic geometric error-correcting code on $\bfX_P$ , obtained by evaluating global section of $\mathcal{L}(D_P)$ on every rational point of $\bfX_P$. This work presents an extension of toric codes analogous to the one of Reed-Muller codes into projective ones, by evaluating on the whole variety instead of considering only points with non-zero coordinates. The dimension of the code is given in terms of the number of integral points in the polytope $P$ and an algorithmic technique to get a lower bound on the minimum distance is described.
	\end{abstract}
	
\keywords{Error-correcting codes, Toric variety, Polytope, Algebraic geometric codes, Gr\"obner bases}

\ccode{Mathematics Subject Classification 2010: 94B05, 14M25, 14J20 , 52B20, 52B05 , 52B11, 13P10.}

\bigskip

		Introduced by J.P. Hansen \cite{HanToric}, toric codes consist in evaluating Laurent monomials $x_1^{m_1}x_2^{m_2}\dots x_n^{m_n}$ at points $(x_1,\dots,x_n) \in (\Fq^*)^n$ where $m=(m_1,\dots,m_n)$ runs through the integral points of a given lattice polytope $P \subset \R^n$. This simple definition hides the algebraic geometric nature of these codes. A lattice polytope $P$ defines a toric variety $\bfX_P$, which contains a dense torus $\T^n\simeq(\overline{\Fq}^*)^n$. The toric code associated to $P$, denoted here by $\C_P$, is the evaluation code of the global sections of the Riemann-Roch space of a divisor $D_P$ on $\bfX_P$ at the rational points of $\T^n$. 
		
		Theoretically, algebraic geometric codes have excellent parameters but their pratical use requires the computation of Riemman-Roch spaces. Arithmetic \cite{Volcheck,Hess} and geometric \cite{Goppa83,LBR,Hache,LGS} algorithms have been developed and refined to compute the Riemman-Roch space of a divisor $D$ on a curve $C$. However, the literature is undoubtably sparser on higher dimensional varieties. Toric varieties present the extremely convient property of having an explicit description of the set of their divisors and their Riemann-Spaces in combinatorial terms, in any dimension. This explains why the parameters of toric codes have been broadly investigated, on surfaces \cite{HanToric,Little,SS} and in higher dimension \cite{LittlemDim,Ruano}.
		
		\smallskip
		
		In this paper, we construct a linear code from the polytope $P$ by evaluating the global sections of the RiemannRoch space $L(D_P)$ at the $\Fq$-points of the whole variety $\bfX_P$, not only on the torus as in toric codes. In the present work, we call the code defined from the polytope $P$ the \emph{projective toric code} asociated to $P$ and we denote it by $\PC_P$. This termininology has be chosen because the difference between a toric code $\C_P$ and its projective version $\PC_P$ is comparable the one between a Reed-Muller code and the projective Reed-Muller code of the same degree \cite{Lachaud}. In the latter case, we evaluate \emph{homogeneous} polynomials in $n+1$ variables instead of polynomials in $n$ variables. Actually, toric varieties, like projective spaces, come with a homogenization process which turns Laurent polynomials of $k[t_1^\pm,\dots,t_n^\pm]$ into polynomials of the Cox ring $k[X_1,\dots,X_r]$ ($r > n$). The parameters of $\PC_P$ compared to those of $\C_P$ behave the same way as for (projective) Reed-Muller codes : for $q$ large enough, the dimension stays the same but the length and the minimum distance are increased.

		This work presents a generic framework of the study of algebraic codes on toric varieties, in which notably (weighted) projective Reed-Muller codes fit (see \cite{Lachaud,Sor} for parameters). As several toric codes are champion codes \cite{BK13b,BK13,L13}, extending them by evaluating outside the big torus is likely to provide new champion codes (see Sec. \ref{sec:good-codes}). 	Besides, this extension of toric codes will probably find applications in information theory. For instance, J.P. Hansen built a \emph{secret sharing scheme with strong multiplication} based on toric codes \cite{HanSecret} in which, as usual, the number of participants is bounded by the length of the code. Projective toric codes may provide similar secret sharing schemes with more participants and analogous techniques are likely to give the parameters of these schemes.

		When studying classical toric codes, the rational points and the global sections have a simple and explicit description. Here arises the problem of handling the $\Fq$-points of the abstract variety $\bfX_P$.
		Instead of embedding the toric variety into $\bfX_P$ into a possibly large projective space \cite{CN}, we focus on the case where the variety $\bfX_P$ is \emph{simplicial} and thus can be represented as a  \emph{good geometric quotient} of an open affine subset of $\A^r$ under the action of a group $G$, where $r$ is the number of facets of the polytope $P$. This standpoint proved to be effective for the study of Goppa codes on Hirzebruch surfaces \cite{nardi:code}. This not only provides a good grasp on the points of $\bfX_P$ but it also expresses the global sections of $L(D_P)$ as polynomials of $\Fq[X_1,\dots,X_r]$. To evaluate them at the rational points $\bfX_P$, we determine which orbits of $r$-tuples correspond to $\Fq$-points, choose a representative among each of these orbits and finally evaluate naively polynomials at these representatives. Note that a generator matrix can then be constructed thanks to Hermite Normal forms without any knowledge about toric varieties (see Section \ref{ExpConst}).

		The combinatorial properties of $P$ echo the geometric properties of the variety $\bfX_P$ and are related to the parameters of the toric codes $\C_P$ and $\PC_P$. Integral points of the polytope $P$ correspond to monomials forming a basis of the line bundle $L(D_P)$. Regarding the dimension, there is a correspondance between a basis of the toric code $\C_P$ and the lattice points of $P$ modulo $q-1$ \cite{Ruano}. Using the decomposition of $\bfX_P$ as a disjoint union of tori, one for each face of $P$, we prove here that reducing the lattice points of $P$ modulo $q-1$ \emph{face by face} provides a basis of $\PC_P$ (Th. \ref{DIM}).
	 		
		To bound the minimum distance from below, we aim to bound from above the number of zeroes in $\bfX_P$ of a global section $f \in L(D_P)$, as usual in the context of algebraic-geometric code. Here, we use a footprint bound technique \cite{GH,Beelen2019} generalizing the one set for codes on Hirzebruch surfaces \cite{nardi:code}. We use Gr\"obner basis theory to express a bound on the minimum distance in terms of integral points in a polytope (Th. \ref{DIST}), without suffering from the exponential growth in the number of variables of complexity of the actual computation of a Gröbner basis.
		
		Altogether, constructing the projective toric code $\PC_P$ and estimating its parameters heavily relies on the determination of lattice points of some polytopes. For a polytope of dimension $N$ of volume $V$ with $s$ vertices, this can be done in time $\tilde{O}\left(\left(s^{\lceil\frac{N}{2}\rceil} + V\right) \log \delta\right)$ where $\delta$ is the maximum modulus of the coordinates of the vertices of $P$ \cite[Prop. 3.5]{SV13}. This encourages a thourough study of (projective) toric codes assocaited to polygons before investigationg in higher dimension.

			\section{Toric variety from a polytope}
			
			This section sums up different definitions and properties of toric varieties defined by polytopes, from the reference book of D. Cox, J. B. Little, and H. K. Schenck \cite{CoxToric}.
			
			\subsection{Normal fan of $P$}
			
			Let $P \subset \R^N$ be a full dimensional convex lattice polytope. Throughout this paper, all polytopes are assumed \emph{convex}. We write $Q \preceq P$ if $Q$ is a face of $P$. For each ${k \in \{0,\dots, N\}}$, the set of $k$-dimensional faces of $P$ is denoted by $P(k)$. The faces of dimension $N-1$ are called \emph{facets} and those of dimension $0$ \emph{vertices}.

			\begin{definition}[Primitive normal vector]\label{def:InnerNormal}
				For each facet $F$ of $P$, denote by $u_F$ the shortest inner normal vector of $F$ with integer coordinates. We call $u_F$ the \emph{primitive normal vector} of the facet $F$.
			\end{definition}
			
			For every face $Q \preceq P$, we define the cone $\sigma_Q = \Cone(u_F \mid Q \preceq F ) \subset \R^N$. The union of these cones forms the \emph{normal fan} of $P$ defined as $\Sigma:= \bigcup_{Q \preceq P} \sigma_Q.$ \cite[Th. 2.3.2]{CoxToric}

			Such a fan defines an abstract toric variety $\bfX_P$, by gluing some affine spaces corresponding to these cones	\cite[Prop. 3.1.6]{CoxToric}.
			
			\begin{fact}\label{IsomVar}
				The toric varieties associated to two polytopes $P$ and $P'$ are isomorphic if and only if their normal fans are equal up to a unimodular transformation, \textit{i.e.} multiplication by a matrix of $\GL_N(\Z)$.
			\end{fact}

			\subsection{Simplicial toric variety as quotient}
			
			Let $\Fq$ be a finite field with $q$ elements, where $q$ is a prime power. \textbf{From now on, we assume that both of the following hypotheses hold.}
				
			\begin{enumerate}[label=(H\arabic*),left= 0pt]
				\item\label{H1} The polytope $P$ is simple: each vertex belongs to exactly $N$ facets.
			\end{enumerate}	f
			For each vertex $v \in P(0)$, let us define the $N$-square matrix $A(v)$ whose rows are the primitive normal vectors $u_F$ (Def. \ref{def:InnerNormal}) of the $N$ facets $F$ containing $v$.

			\begin{enumerate}[resume,label=(H\arabic*),left= 0pt]
				\item\label{H2}The determinant of $A(v)$ is coprime with the characteristic of the field $\Fq$.
				\end{enumerate}

			\begin{remark}\label{rk:H1}
				\begin{enumerate}
					\item \ref{H1} is only meaningful when $N\geq 3$. Every convex polygon is simple. In dimension $3$, cubes and tetrahedra are simple but square-based pyramids are not, as their top vertex belongs to $4$ faces.
					\item Note that \ref{H2} does not depend on the order of the rows of $A(v)$. Moreover, it only matters when the variety given by the polytope $P$ is not smooth, otherwise all these determinants are equal to $1$.
				\end{enumerate}
			\end{remark}

			These hypotheses ensure that the toric variety $\bfX_P$ is \emph{simplicial} and enable us to consider it as the quotient of an open affine subset modulo a group action. More precisely, we see $\bfX_P$ in the affine space $\A^r$, where $r$ is the number of facets of $P$. In the polynomial ring,
			\begin{equation}\label{CoxRing}
				R:=\Fq[X_F \mid \: F \in P(N-1)],
			\end{equation}
			we consider the \emph{irrelevant ideal} $B:=\langle \prod_{F \not\ni v} X_F \mid v \in P(0)\rangle$. We denote by $\mathcal{Z} \subset \A^r$ its set of zeroes. 
			
			\begin{theorem}{\cite[Prop. 5.1.9]{CoxToric}}
				The variety $\bfX_P$ is isomorphic to the quotient of $\A^r \setminus \mathcal{Z}$ under the action of the algebraic group 
				\begin{equation}\label{def:G2}
				G=\left\{(t_F)_{F \in P(N-1)}\in (\overline{\Fq}^*)^r : \forall \: j \in \{1\dots,N\}, \: \prod_{\mathclap{F \in P(N-1)}} t_F^{u_F^{(j)}} = 1\right\}
				\end{equation}
				where $u_F^{(j)}$ is the $j$-th coordinate of the vector $u_F\in \Z^N$, with the following action:
				\[ \left\{\begin{array}{rccl}
					(\A^r \setminus \mathcal{Z})\times G& \rightarrow &\A^r \setminus \mathcal{Z}&\\
					(\bfx, \bft) & \mapsto & \bft \cdot \bfx & \hspace*{-0.6em}:= (t_Fx_F)_{F \in P(N-1)}.
				\end{array}\right.\]
			\end{theorem}

			\begin{remark}
				Proposition 5.1.9 \cite{CoxToric} only holds in characteristic zero. It follows from the reductivity of the group $G$ \cite{guilbot2014}, which is guaranteed here by \ref{H2}.
			\end{remark}
			
			In this presentation of $\bfX_P$, rational points correspond to Frobenius-invariant $G$-orbits, \textit{i.e.} elements $\bfx \in \A\setminus \mathcal{Z}$ such that there exists $\bft \in G$ satisfying $\bfx^q=\bft \cdot \bfx$.

			\subsection{Cox ring and polynomials}
			
			Beside an elegant description of the abstract variety $\bfX_P$, the quotient representation provides a polynomial coordinate ring to the variety $\bfX_P$ as the ring $R$ (Eq. \eqref{CoxRing}) endowed with a grading by the Picard group of $\bfX_P$.
			
			\subsubsection{Picard group of $\bfX_P$}	
			
			The Picard group of the variety $\bfX_P$, denoted by  $\Pic \bfX_P$, is the set of its divisors modulo linear equivalence. We denote by $[D]\in \Pic \bfX_P$ the Picard class of a divisor $D$ on $\bfX_P$.
			The $r$ facets $F \in P(N-1)$ give $r$ torus-invariant prime divisors $D_F$ on $\bfX_P$. Their classes $[D_F]$ generate $\Pic \bfX_P$ as $\Z$-module \cite[Th. 4.1.3]{CoxToric}.
		
			More precisely, 
			\[\Pic \bfX_P \simeq \Z^{r-N} \oplus \bigoplus_{i=1}^N \faktor{\Z}{d_i \Z}\]
			where $1 \leq d_1 \mid d_2 \mid \dots \mid d_N$ are the \emph{invariant factors} of the $r \times N$-matrix $A_P$ whose rows are the primitive normal vectors $u_F$. If $d_i=0$, then $\faktor{\Z}{d_i \Z}=\{0\}$.

			\subsubsection{Cox ring}\label{SecCoxRing}
			
			In the polynomial ring $R=\Fq[x_F \mid F \in P(N-1)]$, we define the \emph{degree} of a monomial $M=\prod_{F \in P(N-1)} X_F^{\alpha_F}$ as the Picard class $\left[\sum a_F D_F\right] \in\Pic \bfX_P$.
			Denote by $R_\alpha$ the $\Fq$-vector space spanned by monomials of degree $\alpha \in \Pic \bfX_P$. Then
			\[R=\bigoplus_{\mathclap{\alpha \in \Pic \bfX_P}} R_\alpha.\]
			
			\begin{fact}\cite[Th 5.4.1]{CoxToric}
				The component $R_\alpha$ is isomorphic to the Riemann-Roch space $L(D)=H^0(\bfX_P, \mathcal{O}(D))$ of any divisor such that $[D]=\alpha$.
			\end{fact}
			
			Riemann-Roch spaces come with a handy depiction thanks to a correspondence between effective divisors and polytopes. For a divisor $D=\sum \alpha_F D_F$, we define a polytope $P_D$ as follows:
			\begin{equation}\label{polydiv}
				P_D=\left\{m \in \R^N \mid \forall \: F \in P(N-1), \: \ps{m}{u_F} \geq -\alpha_F\right\}.
			\end{equation}
			For $m \in P_D$, set $\chi^\ps{m}{P_D}$ the monomial in $R$ defined by
			\begin{equation}\label{monhom}
				\chi^\ps{m}{P_D}=\prod X_F^{\ps{m}{u_F}+\alpha_F}.
			\end{equation}
			Then $L(D)=\Span\left\{\chi^\ps{m}{P_D}, \: m \in P_D \cap \Z^N\right\}$.

			\section{Explicit construction of the projective toric code associated to a polytope}	
			
			\subsection{Definition of $\PC_P$}

			\begin{definition}\label{DefCode}
				Let $P$ be a lattice polytope and $\bfX_P$ its corresponding toric variety. Call $D$ the divisor such that $P=P_D$ \eqref{polydiv}.
				Choose a set $\mathcal{P}$ of representatives of $\Fq$-rational points of $\bfX_P$ as points of $\A^r$. We define the \emph{projective toric code} associated to $P$, that we denote by $\PC_P$, as the image of the evaluation map
				\[\ev_P:\left\{\begin{array}{rcl}
					L(D) & \rightarrow & \Fq^{n}\\
					\chi^\ps{m}{P} & \mapsto & \left(\chi^{\ps{m}{P}}(\bfx)\right)_{\bfx \in \mathcal{P}}
				\end{array}\right.\] 
				where $n:=\#X_P(\Fq)$ is given in \eqref{eq:NbPts}.
			\end{definition}
			
			\begin{remark}
				One can easily check that, given two $r$-tuples $\bfx$, $\bfx'$ such that there exists $\bft \in G$ satisfying $\bfx'=\bft \cdot \bfx$, we have
				\[\chi^{\ps{m}{P}}(\bfx')= \: \prod_{\mathclap{F \in P(N-1)}} t_F^{\alpha_F}\times \chi^{\ps{m}{P}}(\bfx).\]
				Therefore, taking a different set a of representatives of the $\Fq$-points of $\bfX_P$ gives an Hamming-equivalent code.
			\end{remark}

			\paragraph{Why \emph{projective} toric codes?} The classical toric code \cite{HanToric} associated to a polytope $P$ is defined by
			\begin{equation}\label{def:toric-codes}
				\C_P= \Span \{(\chi^m(\bft) \mid \bft \in (\Fq^*)^N), \: m \in P\cap \Z^N \}, \text{ where } \chi^m(t_1,\dots,t_N)=\prod_{i=1}^N t_i^{m_i}.
			\end{equation}
			It is the evaluation of some regular functions $\chi^m$ on the dense torus $\T^N$, called \emph{characters}. To get from toric codes to projective ones, we \emph{homogenize} these characters into the monomials $\chi^\ps{m}{P}$ \eqref{monhom} in the Cox ring $R$, which we evaluate at every $\Fq$-rational point of $\bfX_P$, not only at points on the torus. This process is actually the same that turns a Reed-Muller code into a projective one, hence the terminology \emph{projective} toric codes chosen here.

			\subsection{(Rational) points of $\bfX_P$}
			
			The main issue arising from Def. \ref{DefCode} when handling the code $\PC_P$ is the determination of a set $\mathcal{P}$ of representatives of the $\Fq$-rational points of $\bfX_P$ in its quotient representation. For instance, in the projective space $\PP^N$, we are used to represent a point by a $(N+1)$-tuple with its far-left (or far-right) non-zero coordinate equal to $1$. We aim to generalize this on other simplicial toric varieties by choosing representatives in a \emph{normalized form}, \textit{i.e.} as tuples with some determined zero coordinates and as many coordinates as possible equal to 1. Let us first fix some notations.

			\begin{notation}\label{PowMat}
				Let $A$ be a matrix of size $\ell\times m$ with integer entries. For any $\ell$-tuple $\bft$ with coordinates in a field $\K$, we write $\bft^A$ the $m$-tuple whose $j$-th coordinate of $\bft^A$ is equal to $\prod_{i=1}^\ell t_i^{a_{ij}}$.
			\end{notation}
			
			For any couple of matrices $A, \: B$ with compatible sizes, we have $(\bft^A)^B=\bft^{AB}$.
			
			This notation and Eq. \eqref{def:G2} gives a new definition of the group $G$:
			\begin{fact}\label{fact:G}
				Set $\bfun_N=(1,\cdots,1) \in \Z^N$. The group $G$ is the preimage of $\{\bfun_N\}$ under the map $\bft \in (\overline{\Fq}^*)^r \mapsto\bft^{A_P}\in(\overline{\Fq}^*)^N$.
			\end{fact}
			
			While handling $\bfX_P$ as a quotient space under the action of $G$, we will be often led to exhibit the existence of preimages under  some maps $\bft \mapsto \bft^A$ for some square matrices $A$ by using the following lemma.
			
			\begin{lemma}\label{solsys}
				Let $A$ be a square integer matrix of size $\ell$. If the determinant of $A$ is coprime to the characteristic of the field $\K$, then the map
				\[
				\left\{\begin{array}{rcl}
					\overline{\Fq}^\ell & \rightarrow & \overline{\Fq}^\ell \\	
					\bft & \mapsto  & \bft^A
				\end{array}\right.
				\]
				is surjective.
				
			\end{lemma}
			
			\begin{proof}
				The inverse of $A$ is equal to the transpose of its cofactor matrix, whose entries are integers, divided by its determinant. Therefore, finding a solution of $\bft^A=\bfs$ is possible if we are able to find roots of order $\det(A)$ of any element of $\overline{\Fq}$, which is possible if the determinant of $A$ and the characteristic of $\Fq$ are coprime.
			\end{proof}

			\subsubsection{Local coordinates}
			
			The toric vatiety $\bfX_P$ is covered by $\#P(0)$ affine toric charts corresponding to the vertices of the polytope $P$: given a vertex $v \in P(0)$, define the affine open set of $\bfX_P$
			\begin{equation}\label{eq:U_v}
				\mathcal{U}_{v}:=\{ x_F\neq 0 \mid F \not\ni v\} \subset \A^r \setminus \mathcal{Z}.
			\end{equation}

			\begin{lemma}\label{RepOrbit}
				The $G$-orbit of a geometric point of $\mathcal{U}_{v}$ contains a $r$-tuple $(x_F)_{F \in P(N-1)}$ such that $x_F = 1$ for $F \not\ni v$.
			\end{lemma}
			
			\begin{proof}
				Fix $\bfx=(x_F) \in \mathcal{U}_{v}$. Let us prove there exists $\bft \in G$ such that $t_F=x_F^{-1}$ if $F \not\ni v$. By Eq. \eqref{def:G2}, we can find such a $r$-tuple $\bft=(t_F)_{F \in P(N-1)}$ in $G$ if there exists a $N$-tuple $(t_F)_{F \ni v}$ such that
				\begin{equation}\label{eqlc}
					\forall \: j \in \{1,\dots,N\}, \: \prod_{F \ni v} t_F^{u_F^{(j)}} = \prod_{F \not\ni v} x_F^{u_F^{(j)}}.
				\end{equation}		
				Take $A(v)$ the $N$-square matrix whose rows are the vectors $u_F$ of facets $F \ni v$ (see \ref{H1}). We can solve Eq. \eqref{eqlc} by finding a $N$-tuple $\tilde{\bft}$ satisfying $\tilde{\bft}^{A(v)}=\bfy$ where $\bfy$ is the vector formed by the $N$ right-hansides of \eqref{eqlc}, which is possible by Lem. \ref{solsys} and \ref{H2}.
			\end{proof}
			
			\subsubsection{Normalized forms}
			
			The action of the big torus $\T^N$ splits the toric variety $\bfX_P$ into a disjoint union of finitely many {$\T^N$-orbits}, each orbit itself being an algebraic torus:
			\begin{equation}\label{union}
				\bfX_P= \bigsqcup_{Q \preceq P} \T_Q,
			\end{equation}
			where $\T_Q \simeq (\overline{\K}^*)^{\dim Q}$ \cite[Th. 3.2.6]{CoxToric}. The torus $\T_Q$ is well-described in the affine set $\A^r \setminus \mathcal{Z}$: it is the set of points whose coordinates $x_F$ are zero for the facets $F$ containing $Q$. It is included in any affine chart $\mathcal{U}_{v}$ (see Eq. \eqref{eq:U_v}) associated to a vertex $v$ of $Q$ and it is $G$-invariant.
			
			We shall give the points of $\bfX_P$ by determining those on the tori $\T_Q$ as $r$-tuples in $\A^r \setminus \mathcal{Z}$ up to the action of $G$.
			
			\begin{definition}
				Let $Q \preceq P$ be a face of the polytope $P$ and $v \in P(0)$ a vertex of $Q$. A \emph{$v$-normalized representative of a point of $\T_Q$} is a $r$-tuple $\bfx \in \mathcal{U}_{v}$ such that $x_F=1$ for $F \not\ni v$ and $x_F=0$ for $F \supseteq Q$.
				
				We will call $\bfx$ a \emph{normalized representative} of a point $p$ if $Q$ and/or $v$ need not to be specified.
			\end{definition}

			From this definition, some points may have several $v$-normalized representatives, as illustrated in Ex. \ref{ex:bad}. But we will describe the normalized tuples that represent the same point by using Fact \ref{fact:G}.
			
			\begin{figure}[h]
					\centering
			\includegraphics{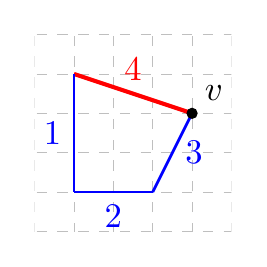}
			\caption{Polygon $P=\Conv((0,0),(2,0),(3,2),(0,3))$.}\label{fig:pol-ex}
			\end{figure}
			
			\begin{example}\label{ex:bad}
				The four inner normal vectors of the quadrilateral $P$ of Fig. \ref{fig:pol-ex} are
		$u_{F_1}=(1,0)$ $u_{F_2}=(0,1)$, $u_{F_3}=(-2,1)$ and $u_{F_4}=(-1,-3)$. The determinants of the matrices $A(v)$ defined in \ref{H2} are $1$, $2$, $7$ and $3$. Let us work on a finite field $\F_q$ with caracteristic different from $2$, $3$ and $7$.
			
				Call $v$ the vertex $(3,2)$. A $v$-normalized representative of a point $p$ on $\T_{F_4}$ has the form $(1,1,x,0)$, since $v$ does not lie on the edges $F_1$ and $F_2$. 
				
				Take $t \in  \overline{\Fq}$ such that $t^7=1$ and $t^3\neq 1$, which is possible by assumption on the characteristic. Then $\bft=(1,1,t^3,t)$ belongs to the group 
				\[G=\{(t_1,t_2,t_3,t_4) \in (\overline{\K}^*)^ 4 \mid t_1=t_3^2 t_4 \text{ and } t_2t_3=t_4^3\},\] which implies that $(1,1,x,0)$ and $(1,1,t^3 x, 0)$ are both some $v$-normalized representatives of a same point of $\T_{F_4}$.
								
			\end{example}

			\subsubsection{Orbits and Hermite normal form}\label{SecHermite}

			\begin{notation}\label{Not}
				Fix a $k$-dimensional face $Q$ of $P$ and a vertex $v$ of $Q$. Since the polytope $P$ is simple (H\ref{H1}), we can number the facets $F_1,\dots,F_r$ of $P$ so that
				\begin{itemize}
					\item $P(v)=\{F_1,\dots,F_N\}$,
					\item $\displaystyle Q = \bigcap_{i=1}^{N-k} F_i$.
				\end{itemize} 
				
				Set $A(v,Q)$ the matrix whose rows are the primitive normal vectors $u_{F_i}$ for ${i \in \{1,\dots,N\}}$. Let $H(v,Q)$ be the \emph{lower triangular Hermite Normal Form} (HNF) of $A(v,Q)$: there exists $T(v,Q) \in \GL_N(\Z)$ that satisfies 
				\begin{equation}\label{hermite}
					A(v,Q) T(v,Q)=H(v,Q).
				\end{equation}
				Let us write the matrix $H(v,Q)$ by blocks as follows:
				\begin{equation}\label{hermiteblocs}
					H(v,Q)=\begin{pmatrix}
						L_1(v,Q) & 0 \\
						B(v,Q) & L_2(v,Q)
					\end{pmatrix}
				\end{equation}
				
				where $L_1(v,Q)$ and $L_2(v,Q)$ are lower triangular square matrices of size $N-k$ and $k$ respectively.
			\end{notation}
			
			\begin{lemma}\label{2rep}
				Using Notations \ref{Not} and \ref{PowMat}, two $r$-tuples $\bfx$ and $\bfx'$ are some $v$-normalized forms of the same point on $\T_Q$ if and only if
				\[\left(x_{F_{N-k+1}},\dots,x_{F_N}\right)^{L2(v,Q)}=\left(x'_{F_{N-k+1}},\dots,x'_{F_N}\right)^{L2(v,Q)}.\]
			\end{lemma}
			
			\begin{proof}
				
				Two tuples $\bfx$ and $\bfx'$ are some $v$-normalized forms of the same point on $\T_Q$ if and only if there exists $\bft \in G$ such that $\bft \cdot \bfx = \bfx'$. This is equivalent to $(t_F)_{F \in P(N-1)} \in G$, with $t_F=1$ for $F \not\ni v$ and $t_F=\frac{x'_F}{x_F}$ for $F \ni v$ such that $F \not\supseteq Q$. There is no condition on the other coordinates $t_F$ since $x_F=x'_F=0$ for $F\supseteq Q$. 
				
				Such a tuple $\bft$ belongs to $G$ if and only if
				\[\left(t_{F_1},\dots,t_{F_{N-k}},\frac{x'_{F_{N-k+1}}}{x_{F_{N-k+1}}}, \dots, \frac{x'_{F_{N}}}{x_{F_{N}}},1,\dots,1\right)^{A_P} = {\bf1}_r\]
				where $A_P$ is the $r \times N$ matrix whose rows are the $u_{F_i}$, ordered in the same way than the facets. The last coordinates of $\bft$ being ones, this is equivalent to
				\[\left(t_{F_1},\dots,t_{F_{N-k}},\frac{x'_{F_{N-k+1}}}{x_{F_{N-k+1}}}, \dots, \frac{x'_{F_{N}}}{x_{F_{N}}}\right)^{A(v,Q)} = {\bf1}_N.\]
				The vector ${\bf1}_N$ being invariant under powering by $T(v,Q)$, this is tantamount to
				\begin{equation}\label{eq:PowerH}
					\left(t_{F_1},\dots,t_{F_{N-k}},\frac{x'_{F_{N-k+1}}}{x_{F_{N-k+1}}}, \dots, \frac{x'_{F_{N}}}{x_{F_{N}}}\right)^{H(v,Q)} = {\bf1}_N,
				\end{equation}
				where $H(v,Q)$ is the HNF of $A(v,Q)$ (Eq. \eqref{hermite}).
				
				\smallskip
				
				It is thus enough to prove that there exists $(t_{F_1}, \dots, t_{F_{N-k}})\in (\overline{\Fq}^*)^{N-k}$ satisfying \eqref{eq:PowerH} if and only if
				\begin{equation}\label{eq:PowerL}
					\left(\frac{x'_{F_{N-k+1}}}{x_{F_{N-k+1}}}, \dots,\frac{x'_{F_{N}}}{x_{F_{N}}}\right)^{L_2(v,Q)} = {\bf1}_k,
				\end{equation}
				where the matrix $L_2(v,Q)$ appears in the bloc form of the HNF $H(v,Q)$ (Eq. \eqref{hermiteblocs}).
				
				The \textquote{if} part is clear by definition of $L(v,Q)$ \eqref{hermiteblocs}. \textit{Conversely}, if we assume that Eq. $\eqref{eq:PowerL}$ holds, computing $(t_{F_1}, \dots, t_{F_{N-k}})\in (\overline{\Fq}^*)^{N-k}$ satisfying Eq. \eqref{eq:PowerH} essentially consists in extracting roots whose degrees are the diagonal entries of $H(v,Q)$, which are coprime with the caracteristic \ref{H2} since $\det A(v,Q) = \pm \det H(v,Q)$.
						\end{proof}

			The rational points of $\T_Q$ viewed in the affine chart $\mathcal{U}_{v}$ are easily characterized.	
			
			\begin{lemma}\label{repTQrat}
				Assume that $\bfx$ is a $v$-normalized representative of a point $p$ on $\T_Q$.
				Then $p$ is $\Fq$-rational if and only if $\left(x_{F_{N-k+1}},\dots,x_{F_N}\right)^{L2(v,Q)} \in (\Fq^*)^k$.
			\end{lemma}
			\begin{proof}
				The point $p$ is rational if and only if $\bfx$ and $\bfx^q$ are both some of its $v$-normalized representatives. By Lem. \ref{2rep}, this is equivalent to 
				\[		\left(x_{F_{N-k+1}}^{q-1},\dots,x_{F_N}^{q-1}\right)^{L2(v,Q)}={\bf1}_k.\]
				The proof is concluded by noticing that \[\left(x_{F_{N-k+1}}^{q-1},\dots,x_{F_N}^{q-1}\right)^{L2(v,Q)}=\left(\left(x_{F_{N-k+1}},\dots,x_{F_N}\right)^{L2(v,Q)}\right)^{(q-1)I_k}\]
				and that $(y_1,\dots,y_k)^{(q-1)I_k}={\bf1}_k$ if and only if $(y_1,\dots,y_k) \in (\Fq^*)^k$.
			\end{proof}

			\subsection{Piecewise toric code} 
			
			To construct the evaluation code, we take the naive evaluation of monomials $\chi^{\ps{m}{P}}$ at the $r$-tuples representing the rational points of $\bfX_P$, described in Lem. \ref{repTQrat}.
			
			\begin{remark} Such $r$-tuples are likely to have coordinates in an extension of $\Fq$, as their computation may require root extractions. At first glance, one may doubt that the evaluation of a monomial at a rational point lies in $\Fq^n$. But, as fairly expected, this is the case as stated in Cor. \ref{EvRat}.
			\end{remark}
		
			Given a proper face $Q \preceq P$, let us have a look at the evaluation of monomials $\chi^{\ps{m}{P}}$ at the points on $\T_Q$. They have some zero coordinates, which makes some monomials vanish.
			
			\begin{lemma}\label{zeroQ}
				Let $\bfx$ be a representative of a point on $\T_Q$. Then a monomial $\chi^{\ps{m}{P}}$ has a non zero evaluation at $\bfx$ if and only if $m$ belongs to the face $Q$.
			\end{lemma}
			
			\begin{proof}
				Since $x_F=0$ for $F \supseteq Q$, the evaluation of $\chi^{\ps{m}{P}}=\prod X_F^{\ps{m}{u_F}+a_F}$ is non-zero if and only if $\ps{m}{u_F}+a_F = 0$ for $F \supseteq Q$, which exactly means that $m \in Q$.
			\end{proof}
			
			The previous lemma and the expression of $\bfX_P$ as union of tori associated to faces of the polytope $P$ (see Eq. \eqref{union}) suggest some links between the projective toric codes associated to the polytope $P$ and the classical toric codes $\C_Q$ associated to the faces $Q$ of $P$. The rest of this paragraph aims to prove of the following result.
			
			\begin{proposition} \label{BlockTor}
				For any $Q \in P(k)$, the puncturing of the code $\PC_P$ at coordinates corresponding to points outside of $\T_Q$ is monomially equivalent to the toric code $\C_Q$.
			\end{proposition}
			
			We recall that two codes $C_1$ and $C_2$ of length $n$ and dimension $k$ are said to be \emph{monomially equivalent} if, given a generator matrix $G_1$ of $C_1$, there is an invertible diagonal matrix $\delta$ and a permutation matrix $\Pi$ of size $n$ such that $G_2=G_1\Delta \Pi$ is a generator matrix for $C_2$.

			To prove Prop. \ref{BlockTor}, we first display a more precise expression of the evaluation of a monomial $\chi^{\ps{m}{P}}$ at a point $p \in \T_Q$ by \emph{straightening} the face $Q$: we determine a unimodular transformation that maps $Q\in P(k)$ into a subspace of dimension $k$ spanned by $k$ vectors of the \emph{canonical basis}. For instance, straightening an edge $Q \in P(1)$ of a polygon makes it horizontal or vertical.
			
			\begin{lemma}[Straightening a face]\label{straight}
				Let $Q \in P(k)$ and take a vertex $v$ of $Q$. Then, the matrix $T(v,Q)$, defined in \eqref{hermite}, satisfies:
				\[\forall m \in Q, \: \exists ! (\tilde{m}_1,\dots,\tilde{m}_k) \in \Z^ k \st (m-v)(T(v,Q)^\top)^{-1} =\sum_{i=1}^k \tilde{m}_i e_{N-k+i},\]
				where $(e_1,\dots,e_N)$ is the canonical basis of $\R^N$.
			\end{lemma}
			
			\begin{proof}
				The matrix $T(v,Q)$ is unimodular: its columns form $\Z$-basis of $\Z^N$. Then, for any $m \in Q$, there exists a unique $N$-tuple $\tilde{m}=(\tilde{m}_1,\dots,\tilde{m}_N) \in \Z^N$ such that 
				\[m-v=\sum_{j=1}^N \tilde{m}_j e_jT(v,Q)^\top.\]
				We want to prove that the $N-k$ first coordinates of $\tilde{m}$ are zero. For $i \in \{1,\dots,N\}$, 
				
				\[\ps{m-v}{u_{F_i}}=\sum_{j=1}^N \tilde{m}_j e_jT(v,Q)^\top\left(e_i A(v,Q)\right)^\top=\sum_{j=1}^N \tilde{m}_j H(v,Q)_{i,j}.\]
				
				The matrix $H(v,Q)$ being lower triangular, the entry $H(v,Q)_{i,j}$ is zero for $j >i$ and diagonal entries are not zero. In addition, the lattice points $m$ and $v$ belong to $Q$, then for $i \in \{1,\dots,N-k\}$,
				\[			0=\ps{m-v}{u_{F_i}}=\sum_{j=1}^i \tilde{m}_j  H(v,Q)_{i,j}.\]
				A simple induction on $i$ proves that $\tilde{m}_i=0$ for $i \in \{1,\dots,N-k\}$. Shifting the tuple $\tilde{m}$ gives the expected expression.
			\end{proof}

			\begin{proposition}\label{propRM}
				Let us take $Q \in P(k)$, a lattice element $m \in Q\cap \Z^N$, a vertex $v$ of $Q$ and $\bfx$ a $v$-normalized representative of a point on $\T_Q$. Then, using Notation \ref{Not} of Section \ref{SecHermite}, there exist $\tilde{m_1},\dots\tilde{m_k} \in \Z$ such that
				\begin{equation}\label{RM}
					\chi^{\ps{m}{P}}(\bfx)=\prod_{j=1}^k y_j^{\tilde{m_j}} \text{ where } (y_1,\dots,y_k)=(x_{F_{N-k+1}},\dots,x_{F_N})^{L_2(v,Q)} \in \left(\overline{\Fq}^*\right)^k.
				\end{equation}
			\end{proposition}
			
			\begin{proof} As $m \in Q$, we have $\ps{m}{u_{F_i}}+a_{F_i}=0$ for $i \in \{1,\dots,N-k\}$. The element $\bfx$ is $v$-normalized, which means that $x_{F_i}=1$ for $i >N$. Then
				\[\chi^{\ps{m}{P}}(\bfx)=\prod_{i=N-k}^N x_{F_i}^{\ps{m}{u_{F_i}}+a_{F_i}}.\]
				
				Since $v \in F_i$ for every $ i\in \{1,\dots,N\}$, we have $\ps{v}{u_{F_i}}=-a_{F_i}$. Therefore, for every $i \in \{1,\dots,k\}$, Lem. \ref{straight} implies that 
				\[\ps{m}{u_{F_{N-k+i}}}+a_{F_{N-k+i}}=\ps{m-v}{u_{F_{N-k+i}}}=\sum_{j=1}^k  \tilde{m}_j \left(H(v,Q)\right)_{N-k+i,j}.\] 
				Moreover, we have $\left(H(v,Q)\right)_{N-k+i,j}=\left(L_2(v,Q)\right)_{i,j}$ since $L_2(v,Q)$ is the lower right block of $H(v,Q)$ \eqref{hermiteblocs}. Then $\left(H(v,Q)\right)_{N-k+i,j}=\left(L_2(v,Q)\right)_{i,j}$ and
				\[\chi^{\ps{m}{P}}(\bfx)=\prod_{j=1}^k \prod_{i=1}^k x_{F_{N-k+i}}^{\tilde{m_j} \left(L_2(v,Q)\right)_{i,j}}=\prod_{j=1}^ k y_j^{\tilde{m_j}}.\]
			\end{proof}
			
			\begin{corollary}\label{EvRat}
				The evaluation of a monomial $\chi^{\ps{m}{P}}$ at a normalized representative of a $\Fq$-rational point of $\T_Q$ belongs to $\Fq$.
			\end{corollary}
			
			\begin{proof}
				If $m \notin Q$, then $\chi^{\ps{m}{P}}(x)=0$ by Lem. \ref{zeroQ}. Otherwise, Lem. \ref{repTQrat} states that a $v$-normalized representative of an $\Fq$-rational point on $\T_Q$ satisfies $\bfy=\left(x_{F_{N-k+1}},\dots,x_{F_N}\right)^{L2(v,Q)} \in (\Fq^*)^k$. In this case, the quantity given in Prop. \ref{propRM} clearly lies in $\Fq$.
			\end{proof}
			
			The expression in Prop. \ref{propRM} has exactly the form encountered in toric codes (See Def. \ref{def:toric-codes}). Before precisely stating a comparison between the restriction of the code $\PC_P$ on points of $\T_Q$ and the toric code $\C_Q$, let us recall a classification of toric codes made by J. Little and R. Schwarz \cite{LittlemDim}.
			
			We say that integral polytopes $P_1$ and $P_2$ in $\R^N$ are \emph{lattice equivalent} if there exists an affine transformation $\phi$ defined by $\phi(m) =mT+\lambda$ with $T\in \GL_N(\Z)$ and $\lambda \in \Z^N$ such that $\phi(P_1) =P_2$. Not only lattice equivalent polytopes define isomorphic toric varieties (Fact \ref{IsomVar}), but they also give equivalent codes.
			
			\begin{theorem}\label{MonEquivTor}\cite[Th. 3.3]{LittlemDim}
				If two polytopes $P_1$ and $P_2$ are lattice equivalent, then the toric codes $\C_{P_1}$ and $\C_{P_2}$ are monomially equivalent.
			\end{theorem}
		
		Note that the converse is false (see \cite{LYZZ15} for instance). We have now gathered all the ingredients to prove Prop. \ref{BlockTor}. 
			\begin{proof}[Proof of Prop. \ref{BlockTor}]	
				Fix a vertex $v \in Q$. By Prop. \ref{propRM}, puncturing $\PC_P$ outside $\T_Q$ gives the toric code $\C_{\phi(Q)}$ where $\phi$ is the invertible affine map defined by $\phi(m)=(m-v) \left(T(v,Q)^\top\right)^{-1}$. By Th. \ref{MonEquivTor}, it is thus monomially equivalent to $\C_Q$.
			\end{proof}

			\subsection{A ``generator matrix'' of the code $\PC_P$}\label{ExpConst}
						
			The purpose of this section is to give a matrix $M(P)$ whose rows span the code $\PC_P$. In fact, this matrix represents the evaluation map of Def. \ref{DefCode} in the basis formed by the monomials $\chi^\ps{m}{P}$. Its has $\#(P \cap \Z^N)$ rows and $n$ columns, where 
			\begin{equation}\label{eq:NbPts}
				n:=\#\bfX_P(\F_Q)=\sum_{k=0}^N \#P(k) \times (q-1)^k
			\end{equation}
			can be computed from Eq. \eqref{union}, using that $\T_Q(\Fq) \simeq (\Fq^*)^{\dim(Q)}$. Beware that the matrix $M(P)$ may not have full rank : in this case, it is not strictly speaking a generator matrix of the code.
			
			The depiction of $M(P)$ relies on Prop. \ref{propRM}. Before giving a protocol to construct $M(P)$ for any polytope, let us focus on the simpler case when $P$ is a polygon.

			\subsubsection{Example for polygons ($N=2$)}
			
			We number the facets (which are edges) $F_1,\dots,F_r$ and the vertices $v_1,\dots,v_r$ so that $v_i \in F_{i-1} \cap F_i$. By Prop. \ref{propRM}, the restriction of the code along an edge is a code of Reed-Solomon. Then we can sort the rational points of $\bfX_P$ and the lattice points of $P$ so that the matrix of the linear map $\ev_P$ in the monomial basis has the form given in Fig. \ref{MatEv}, where the grey blocks are Vandermonde-type matrices. For $i \in \{0,1,\dots,r-1\}$, we set
			\[V_i=\begin{pmatrix}
				\alpha_1 & \alpha_2& \dots & \alpha_{q-1}\\
				\alpha_1^ 2 & \alpha_2^ 2 & \dots & \alpha_{q-1}^2\\
				\vdots & \vdots &  &\vdots \\
				\alpha_1^{l_i} & \alpha_2^ {l_i} & \dots & \alpha_{q-1}^{l_i}\\
			\end{pmatrix}\]
			where $\Fq^*=\{\alpha_1,\alpha_2,\dots,\alpha_{q-1}\}$ and $l_i:=\#(F_i^\circ\cap \Z)$ is the number of lattice points that are not vertices on the $i$-th edge.

			\begin{figure}[h]
				\centering
				\includegraphics{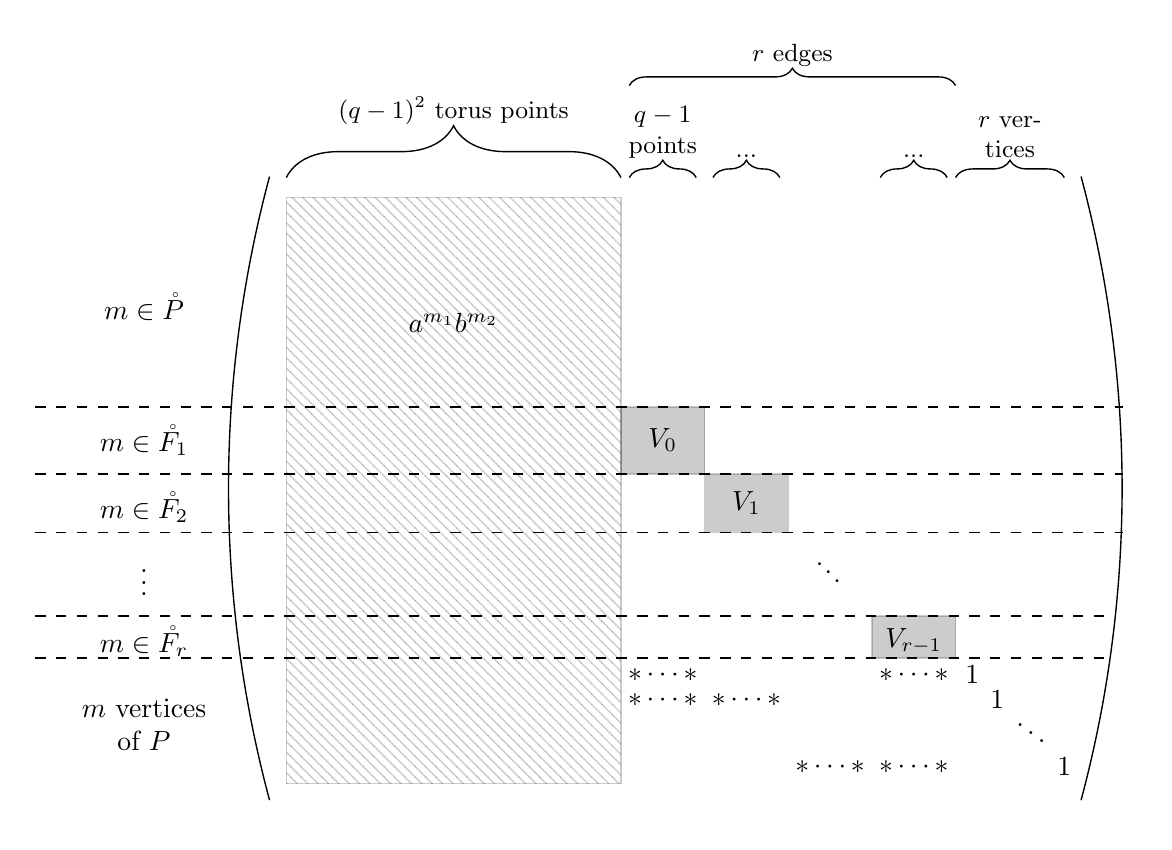}
				\caption{Matrix of the evaluation map associated to a polygon $P$ ($N=2$)}
				\label{MatEv}
			\end{figure}

			\subsection{General case}\label{MP}
			
			To construct the matrix $M(P)$, we benefit from the piecewise toric structure of the code $\PC_P$, displayed in Prop. \ref{BlockTor}. For each face $Q$ of $P$, we need an affine transformation $\phi$ that straightens $Q$, which will give the evaluation on the torus $\T_Q$ a a convenient form. From Sec. \ref{SecHermite}, if $v$ is a vertex of $Q$, the HNF of $A(v,Q)$ provides such a map $\phi$. Instead of running through faces $Q$ and choosing a vertex $v \in Q$ to compute $A(v,Q)$, one idea to save some unnecessary computations of HNFs is to choose flags of faces of $P$, denoted by $\mathcal{F}=(Q_{0} \prec Q_{1} \prec \dots \prec Q_{N-1}\prec P)$, with $\dim Q_{j} = j$ ($i=1,\dots,b$) so that we have one matrix $A_\mathcal{F}$ satisfying $A_\mathcal{F}=A(Q_0,Q_j)$ for all $j \in \{1,\dots,N\}$.
			Then the transition matrix associated to the Herminite normal form of $\mathcal{A}_\mathcal{F}$ straighten all the faces $Q_j$ at once.

			\begin{itemize}
				\item \textbf{Step 0:} Order the points of $P\cap \Z^N$. See Remark \ref{Ordre} for a suggestion of an order that gives the resulting matrix $M(P)$ a triangular by blocks form.
				
				\item \textbf{Step 1:} Find $b$ flags of faces of $P$, denoted by $\mathcal{F}_i=(Q_{i,0} \prec Q_{i,1} \prec \dots \prec Q_{i,N-1}\prec P)$ with $\dim Q_{i,j} = j$ ($i=1,\dots,b$) so that every face $Q \preceq P$ belongs to one of these flags. We can choose $b=\max_{j} \#P(j)$. 
				\item \textbf{Step 2:} Fix $i \in \{1,\dots,b\}$. For every $j \in \{0,\dots,N-1\}$, call $E_{i,j}$ the underlying vector space of $Q_{i,j}$, centered in the point $Q_{i,0}$. We want to determine an invertible integer affine map $\phi_i$ such that 
				\begin{equation}\tag{$\star$}\label{GoodStraight}
					\phi_i(E_{i,j})=\Span(e_1,\dots,e_j).
				\end{equation}
				Set $A_i$ the $N$-square matrix whose rows are the normal vector $u_{F_j}$ of the N facets containing the vextex $Q_{i,0}$, ordered so that $Q_{i,j}=\bigcap_{\ell=0}^{N-j} F_\ell$ for every $j \in \{0,\dots,N\}$. Let $H_i$ be the HNF of $A_i$ and $T_i \in \GL_N(\Z)$ the unimodular matrix such that $A_i T_i= H_i$. Denote by $T'_i$ the matrix obtained by reversing the order of the columns of $T_i$. Then the affine map $\phi_i(m)=T_i'(m-Q_{i,0})$ satisfies \eqref{GoodStraight} by a 
				similar\footnote{We just reverse the columns of $T_i$ for \eqref{GoodStraight} to hold. Otherwise, we would have $\phi_i(E_{i,j})=\Span(e_{N-j+1},\dots,e_N)$ as in Lem. \ref{straight}.} 
				argument than in the proof of Lem. \ref{straight}.
				
				\item \textbf{Step 3:} Now fix a face $Q \in P(k)$. Then there exists $i \in \{1,\dots,b\}$ such that $Q=Q_{i,k}$. 
				
				For every $m \in P \cap \Z^N$, compute the vector $w(m,Q) \in (\Fq)^{(q-1)^k} $ defined by
				\begin{equation}\label{mot}
					w(m,Q)=\left\{\begin{array}{cl}
						\left(\prod_{j=1}^k {x_j}^{\phi_i(m)^{(j)}} \mid (x_1,\dots,x_k) \in (\Fq^*)^k\right) & \text{if } m \in Q, \\
						(0,\dots,0) & \text{if } m \notin Q,
					\end{array}
					\right.
				\end{equation}
				where $\phi_i(m)^{(j)}$ is the $j$-th coordinate of the vector $\phi_i(m)$. By construction of $\phi_i$, if $m \in Q$, then $\phi_i(m)^{(j)}=0$ for $j \in \{k+1,\dots,N\}$.
				
				Form the matrix $M(P,Q)$, of size $\#(P \cap \Z^N)\times (q-1)^k$ by stacking the vectors $w(m,Q)$ on top of each other, in the order chosen at Step 0.
				\item \textbf{Step 4:} Form the matrix $M(P)$ by putting side by side the matrices $M(P,Q)$ constructed in the previous step.
			\end{itemize}

			\begin{remark}\label{Ordre}
				Similarly to the case $N=2$ (See Fig. \ref{MatEv}), we can order the lattice points and the faces of $P$ to make the matrix $M(P)$ triangular by blocks.
				
				First put the block $M(P,P)$, followed by the blocks $M(P,F_i)$ corresponding to facets, then those corresponding to faces of dimension $N-2$ and so on, making the dimension of the faces decrease:
				
				\begin{center}
					\begin{tabular}{C{13mm}C{18mm}C{18mm}C{8mm}C{15mm}C{8mm}C{20mm}}
						$M(P)=$&$\Big(M(P,P)$ &$M(P,F_i)$&$\cdots$& $M(P,Q_i^k)$& $\cdots$& $M(P,v_i)\Big)$\\
						&$\underbrace{\phantom{M(P,P)}}$&$\underbrace{\phantom{M(P,P)}}$& &$\underbrace{\phantom{M(P,P)}}$& &$\underbrace{\phantom{M(P,P)}}$\\
						& big torus& facets & $\cdots$& $k$-dim. faces  & $\cdots$ & vertices
					\end{tabular}
				\end{center}
				
				In the same way, we order the lattice points of $P$ at Step 0 starting by points in the interior of $P$, then in the interior of facets and so on, ending with the vertices.
				
				This way, given a face $Q \preceq P$, the first rows of $M(P,Q)$ correspond to evaluations of $\chi\ps{m}{P}$ on $\T_Q$ for lattice points $m$ inside faces $Q'\neq Q$ with $\dim Q' \geq \dim Q$. These rows are thus filled with zeros, by definition of $w(m,Q)$ (Eq. \eqref{mot}).
			\end{remark}

			\section{Dimension}
			
			As for classical toric codes, the computation of the dimension of a projective toric code $\PC_P$ relies on the reduction modulo $q-1$ of the integral points of the polytope $P$.
			
			\subsection{Reduction modulo $q-1$ and dimension}
			
			\begin{definition}\label{DefEquiv}
				Let us a define a equivalence relation on $\Z^N$ as follows: for two elements $(u,v)\in (\Z^N) ^2$, we write $u \sim v$ if $u-v \in (q-1) \Z^N$.
				
				Let $U$ be a subset of $\Z^N$. We will call a \emph{reduction of} $U$, denoted by $\overline{U}$, a set of representatives of $U$ under the equivalence relation $\sim$, that is to say a subset of $U$ with the following property:
				\[\forall u \in U, \: \exists ! \: \overline{u} \in \overline{U} \text{ s.t. } u \sim \overline{u}.\]
				If $P$ is a lattice polytope in $\R^N$, we denote by $\overline{P}$ a reduction of the lattice points of $P$, i.e. $\overline{P}=\overline{P \cap \Z^N}$.
			\end{definition}

			Reducing the lattice points of the polytope $P$ modulo $q-1$ provides a basis of a classical toric code \cite{Ruano}.

			\begin{theorem}\label{DimTor}\cite{Ruano}
				Let $P$ be a polytope and its associated toric code $\C_P$. Take $\overline{P}$ a reduction of $P$ modulo $q-1$. The kernel of the evaluation map is the $\Fq$-vector space spanned by
				\[\{\chi^m - \chi^{\overline{m}}, \: m \in (P\cap \Z^n)^2, \: \overline{m} \in \overline{P} \st m \neq \overline{m} \}.\]
				A basis of the code $C_P$ is
				\[\left\{(\chi^{\overline{m}}(\bft), \: \bft \in \T^N) \mid \overline{m} \in \overline{P}\right\} \]
				and therefore the dimension of $\C_P$ is equal to $\#\overline{P}$.
			\end{theorem}

			We benefit from the form of the matrix $M(P)$, constructed in Sec. \ref{MP}, to use this result on classical toric codes in order to get the dimension of projective ones. We are led to reduce the polytope $P$ \emph{face by face}.
			
			\begin{definition}
				For a polytope $P$, we define the interior of $P$ as the set of points of $P$ that do not belong to any proper face of $P$, which we denote by
				\[P^\circ=\{m \in P \st \forall Q \preceq P \text{ with } Q \neq P, \: m \notin Q\}.
				\]
			\end{definition}
			
			\begin{definition}\label{DefEquivProj}
				Given a lattice polytope $P \subset \Z^N$, we define an equivalence relation $\sim_P$ on the set of its lattice points $P \cap \Z^N$ by
				\[m \sim_P m' \: \Leftrightarrow \: \exists Q \preceq P \st m, \: m' \in Q^\circ \AND m-m' \in (q-1) \Z^N.\]
				We call a \emph{projective reduction of $P$} any set of representatives of elements of $P \cap \Z^N$ modulo $\sim_P$.
			\end{definition}
			
			We will denote a projective reduction of $P$ by $\Red(P)$ to emphasize the difference between the relation defined in Def. \ref{DefEquiv} in the study of classical toric codes and in the one of Def. \ref{DefEquivProj} that we use here. However, given reductions of the interior of every face $Q \preceq P$, we get a projective reduction of $P$ by setting $\Red(P) = \bigcup_{Q \preceq P} \overline{Q^\circ}$.
			
			\begin{theorem}\label{DIM} Let $P$ be a polytope and $\PC_P$ its associated toric code.  Let $\Red(P)$ be a projective reduction of $P$. Then
				\[\ker \ev_P = \Span \{\chi^m - \chi^{\overline{m}}: \: m \in (P\cap \Z^n)^2, \: \overline{m} \in \Red(P) \st m \neq \overline{m} \}\]
				and the set 
				\[\left\{\ev_P(\chi^\ps{\overline{m}}{P}) \mid \overline{m} \in \Red(P)\right\} \]
				forms a basis of the code $\PC_P$.
				The dimension of the code $\PC_P$ is equal to $\dim \PC_P = \#\Red(P)$.
			\end{theorem}
			
			\begin{proof}
				The dimension of $\PC_P$ is the rank of the matrix $M(P)$, defined in Section \ref{MP}. By Remark \ref{Ordre}, it is the sum of the rank of the matrices $\tilde{M}(P,Q)$. By Prop. \ref{BlockTor}, the rank of $\tilde{M}(P,Q)$ is the dimension of the toric code $C_{Q^\circ}$, which is equal to  $\#\overline{Q^\circ}$, by Prop. \ref{DimTor}.
			\end{proof}
			
			\begin{remark}\label{beware}
				Beware that Prop. \ref{DIM} does not provide the dimension of \emph{any} algebraic geometric code assodicated to a divisor $D$ on $\bfX_P$. The result holds if and only if the polytope associated to $D$ has the same normal fan than $P$, which is equivalent for $D$ to be very ample on $\bfX_P$. If this is not the case, the kernel may not contain only binomials (see \cite{nardi:code} for example on Hirzebruch surfaces).
			\end{remark}
			
			\begin{example}
				Let $a$, $b$ and $\eta$ be three positive integers. Consider the quadrilateral $P$ whose vertices are $(0,0)$, $(a,0)$, $(a,b)$ and $(0,b+\eta a)$ (See Fig. \ref{Hirz}). Such a trapezoid describe a toric surface parametrized by the integer $\eta$ called a Hirzebruch surface (See Fig. 2 \cite{HanToric}). 
				
				The black dots are the lattice points of $\overline{Q^\circ}$ for every face $Q$ of $P$. Figures \ref{Hirz1} and \ref{Hirz2} respectively display points of $\overline{P^\circ}$ and those of the interior of edges of $P$. Figure \ref{Hirz3} shows the final reduction face by face of $P$, when adding the vertices to the two previous steps.
				
				The number of black dots is thus equal to the dimension of the code on $\F_7$, whose explicit formula depending on $(a,b,\eta)$ is given by \cite[Th. 2.4.1]{nardi:code}.
				
				\begin{figure}[h]
					\centering
					\begin{subfigure}[b]{0.3\textwidth}
						\centering
					\includegraphics{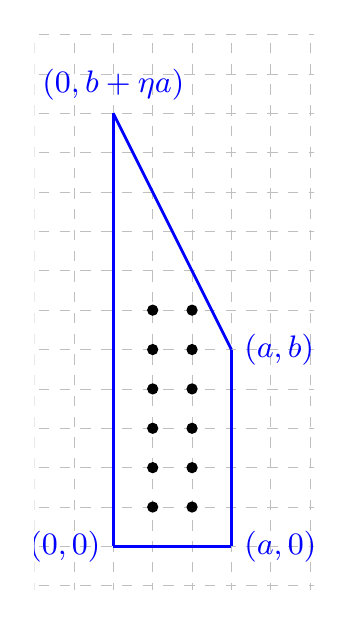}
						\caption{Reduction of $P^\circ$}\label{Hirz1}
					\end{subfigure}
					\begin{subfigure}[b]{0.3\textwidth}
						\centering
						\includegraphics{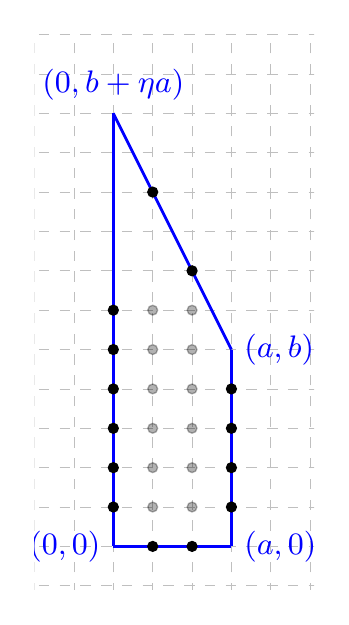}
						\caption{Reduction of edges}\label{Hirz2}
					\end{subfigure}%
					\begin{subfigure}[b]{0.3\textwidth}
						\centering
						\includegraphics{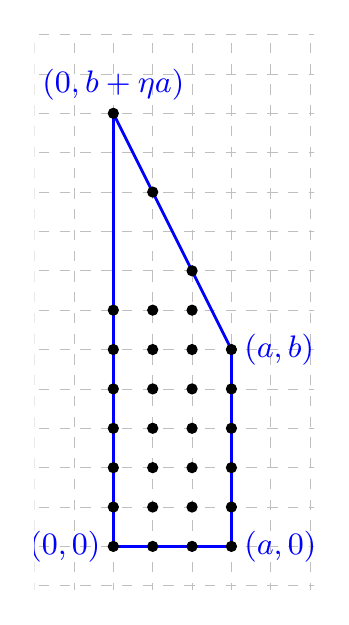}
						\caption{Total reduction of $P$}\label{Hirz3}
					\end{subfigure}%
					\caption{Reduction face by face of a polygon associated to a Hirzebruch surface $\mathcal{H}_2$ for $q=7$}\label{Hirz}	
					
				\end{figure}
			\end{example}
			
			\section{Gr\"obner basis theory in the Cox Ring}
			
			To compute a lower bound for the minimum distance, we use a technique similar to the one used in the study of Hirzebruch surfaces \cite{nardi:code}: we estimate the dimension of a quotient of vectors spaces. We shall take advantage of the graded structure of the Cox ring and use Gr\"obner basis theory, which may be powerful in this framework as it provides a basis of the quotient of a ring modulo an ideal.
			Let us first recall some facts about Gr\"obner bases (See \cite{Cox} for details).
			
			Let $R$ be a polynomial ring. Denote by $\M$ the set of monomials of $R$. A \emph{monomial order} is a total order on $\M$, denoted by $<$ that is compatible with the multiplication and such that $1 < M$ for any $M \in \M$. For every polynomial $f \in R$, we call the \emph{leading monomial} of $f$ and denote by $\LM(f)$ the greatest monomial for this ordering that appears in $f$. 
			
			Let $I$ be an ideal of a polynomial ring $R$ over a field\footnote{We assume $R$ to be a polynomial ring over a \emph{field} to avoid the distinction between leading term and leading monomial.}, endowed with a monomial order $<$. The \emph{monomial ideal} $\LM(I) \subset R$ associated to $I$ is the ideal generated by the leading monomials $\LM(f)$ of all the polynomials $f \in I$.  A finite subset $G$ of an ideal $I$ is a \emph{Gr\"obner basis} of the ideal $I$ if $\LM(I)=\left\langle \LM(g) \: | \:  g \in G \right\rangle$.
			
			The pleasing property of Gr\"obner bases that will be used to estimate the minimum distance of the code is the following.
			
			\begin{proposition}{\cite[ Prop. 1.1]{Sturm}}\label{Grob}
				Let $I$ be an ideal of a polynomial ring $R$ over a field $\K$ with Gr\"obner basis $G$. Then, setting $\pi$ as the canonical projection of $R$ onto $\faktor{R}{I}$, the set
				\[\{\pi(M), \: M \in \M \text{ s.t. } \forall \: g \in G, \: \LM(g) \mid M\}\]
				is a basis of $\faktor{R}{I}$ as a $\K$-vector space.
			\end{proposition}
			
			\subsection{Application in the Cox ring}
			
			Let us set a monomial order on the Cox Ring of $\bfX_P$ by using its grading by $\Pic \bfX_P$ and the correspondence between lattice points in $P_D$ and element of $R_{[D]}$ (See Par. \ref{SecCoxRing}).
			
			\begin{notation}\label{not:order}
				Fix a total order on $\Pic \bfX_P$ and one on $\Z^N$ that is compatible with the addition, both denoted by $<$, from which we define a total order $<$ on $R$ as follows.
				
				Given $D_1$ and $D_2$ two divisors on $\bfX_P$ and $m_1, \: m_2$ two tuples in $\Z^N$ such that $m_i \in P_{D_i}$, we set $\chi^{\ps{m_1}{D_1}} < \chi^{\ps{m_2}{D_2}}$ if one of the following holds:
				\begin{enumerate}
					\item $[D_1] < [D_2]$ in $\Pic \bfX_P$,
					\item $[D_1] = [D_2]$ (then $\exists m \in \Z^n$, such that $P_{D_1}=m+P_{D_2}$) and $m_1 < m+m_2$. 
				\end{enumerate}
			\end{notation}
			
			In the Cox ring $R$, we define $I$ as the homogeneous vanishing ideal in the subvariety consisting of the $\Fq$-rational points of $\bfX_P$. As the Cox ring is graded by the Picard group of the variety, the homogeneous component of $I$ of ``degree'' $\alpha \in \Pic \bfX_P$ is isomorphic to the kernel of $\ev_{D}$, for any $D \in \alpha$.

			Provided a monomial order on $R$, Prop. \ref{Grob} gives a basis of $\faktor{R}{I}$ as $\Fq$-vector space. Taking the elements of degree $[D]$ gives a basis of the homogeneous component $\faktor{R_{[D]}}{I_{[D]}}$. Since 
			\[\faktor{R_{[D]}}{I_{[D]}}\simeq \faktor{L(D)}{\ker \ev_P}\:,\]
			Theorem \ref{DIM} states that a projective reduction of $P$ provides a basis of this same vector space. We would like to make these two bases coincide by choosing a projective reduction of $P$ that takes into account the monomial order we consider.
			
			\begin{definition}
				Let $<$ a total order on $\Z^N$. Let $P$ be a lattice polytope and $m \in P \cap \Z^N$. Define $\Red_<(m,P)$ as the minimum lattice point of $P$ under the order $<$ such that $\Red_<(m,P) \sim_P m$. By extension, we define the \emph{projective reduction of $P$ with respect to the order $<$}:
				\[\Red_<(P)= \{ \Red(m,P), \: m \in P \cap \Z^N\}.\]
			\end{definition}
			
			\begin{example}
				The points displayed in Fig. \ref{Hirz} are in fact the elements of $\Red_<(P)$ for the lexicographic order.
			\end{example}

			\begin{lemma}\label{NiceBG}
				Given an order $<$ on the Cox ring as in Not. \ref{not:order}, there exists a Gr\"obner basis $\mathcal{G}$ of the ideal $I$ such that 
				\[		\{ \chi^\ps{m}{P}, \: m \in \Red_<(P) \}=\{M \in \M_D \st \forall g \in \mathcal{G}, \LM(g) \nmid M\}.
				\]
			\end{lemma}
			
			\begin{proof}
				By Prop. \ref{DIM}, the set 
				\[\mathcal{B}_D:=\{\chi^\ps{m}{P}-\chi^\ps{\Red_<(m,P)}{P}, \: m \in P \st m \neq \Red_<(m,P)\} \]
				forms a a basis of $I_{[D]}\subset I$, where $D$ is the divisor associated to $P$. Take a Gr\"obner basis $\mathcal{G}$ of $I$ that contains $\mathcal{B}_D$. It is always possible: a Gr\"obner basis of $I$
				to which we add element of $I$ remains a Gr\"obner basis of $I$.
				
				The set of monomials $\{ \chi^\ps{m}{P}, \: m \in \Red_<(P) \}$ contains exactly the monomials that are not divisible by any $\LM(g)$ for any $g \in \mathcal{B}_D$. Then, we have 
				\[		\{ \chi^\ps{m}{P}, \: m \in \Red_<(P) \}\supseteq\{M \in \M_D \st \forall g \in \mathcal{G}, \LM(g) \nmid M\}.
				\]
				Both of these sets being some bases of the vector space $\faktor{R_{[D]}}{I_{[D]}}$, they are thus equal.

			\end{proof}
			
			The result brought by the Gr\"obner basis theory involves divisibility relation in the polynomial ring we consider, here the Cox ring of $\bfX_P$. The following lemma takes advantage of the correspondence between monomials in this Cox ring and lattice points of polytopes to translate divisibility of monomials into a relation between the corresponding points in $\Z^N$.
					
			\begin{lemma}\label{CaracDiv}
				Take two divisors $D_1$ and $D_2$ of $\bfX_P$ and their associated polytopes $P_{D_1}=P_1$ and $P_{D_2}=P_2$ (defined in Eq. \eqref{polydiv}). Then, in the Cox ring of $\bfX_P$, for every $m_1 \in P_1\cap \Z^N$ and $m' \in P_2\cap \Z^N$,
				\[\chi^\ps{m_1}{P_1} \mid \chi^\ps{m_2}{P_2} \: \Leftrightarrow \: m_2-m_1 \in P_2-P_1 \]
			\end{lemma}
			
			\begin{proof} Write $D_i=\sum_F a_{i,F} D_F$ $(i=1,2)$. The monomial $\chi^\ps{m_1}{P_1}$ divides $\chi^\ps{m_2}{P_2}$ if and only if for every $ F \in P(N-1)$, we have
				\[\ps{m_1}{u_F}+a_{1,F} \leq \ps{m_2}{u_F}+a_{2,F} \:  \Leftrightarrow \: \ps{m_2-m_1}{u_F} \geq -(a_{2,F}-a_{1,F}).\]
				This exactly means that $m_2-m_1 \in P_2-P_1$ by \eqref{polydiv}.
			\end{proof}	
			
			\section{Minimum distance}
			
			The dimension of the code $\PC_P$ has been related to the lattice points of the polytope $P$ we hope for a similar kind of result for the minimum distance. Here another polytope comes into play.
			
			\begin{definition}\label{DefSurj}
				Let $P$	be a lattice polytope. A lattice polytope $P'$ is said to be \emph{$P-$surjective} if it has the same normal fan than $P$, contains $P$ and $\PC_{P'}=\Fq^{\#\bfX_P(\Fq)}$.
			\end{definition}
			
			Fact \ref{IsomVar} ensures that a $P-$surjective polytope defines the same variety as $P$. Asking a $P-$surjective polytope to contain $P$ enables us to embed $L(D_P)$ into $L(D_{P'})$. Finally, Prop. \ref{DIM} entails that the reduction modulo $q-1$ of the interior lattice points of each $k$-dimensional faces of $P'$ has $(q-1)^k$ elements.
			
			Given a $P-$surjective polytope and an order on $\Z^N$, one can deduce a lower bound of the minimum distance, as stated by the following theorem.
			
			\begin{theorem}\label{DIST}
				Let $P_\surj$ be a $P-$surjective polytope and let $<$ be an order on $\Z^N$ that is compatible with the addition. Then the minimum distance of the code $\PC_P$ satisfies
				\[d(\PC_P)\geq \min_{m \in \Red_<(P)} \# \left((m + P_\surj-P)\cap \Red_<(P_\surj)\right).\]
			\end{theorem}
			
			The proof of Th. \ref{DIST} is splitted in two parts. Lem. \ref{DistQuot} gives a lower bound  of $d(\PC_P)$ in terms of the dimension of a quotient of vector spaces. Lem. \ref{CdtDiv} benefits from results of the previous section about Gr\"obner bases to handle this dimension as number of points in the polytope $P_\surj$.

			\begin{lemma}\label{DistQuot}
				Call $D$ (resp. $D_\surj$) the divisor on $\bfX_P$ such that $P=P_D$ \eqref{polydiv} (resp. $P_\surj=P_{D_\surj}$).
				
				For any $f \in L(D) \setminus \ker \ev_P$, we define 
				\[\widetilde{n_f}:= \dim \left(\faktor{L(D_\surj)}{\ker \ev_{P_\surj} + F \cdot L(D_\surj - D)}\right).\]
				Then the minimum distance $\displaystyle d(\PC_P)$ satisfies 
				$\displaystyle{n-d(\PC_P) \leq \max_{\substack{f \in L(D) \\ f \notin  \ker \ev_P}} \widetilde{n_f}}$.
			\end{lemma}
			\begin{proof}
				As usual in the framework of algebraic geometric codes, we can write that
				\[d(\PC_P) = n - \max_{\mathclap{\substack{f \in L(D) \\ f \notin  \ker \ev_P}}} \: n_f\]
				where $n_f=\# \mathcal{Z}(f)(\Fq)$ is the number of $\Fq$-points in the zero set $\mathcal{Z}(f)$ of $f$ in $\bfX_P$.
				
				Fix a polynomial $f \in L(D) \setminus \ker \ev_P$. The map
				\[\ev_{P_\surj,f}: \left\{ \begin{array}{rcl} L(D_\surj) & \rightarrow & \Fq^{n_f} \\
					g & \mapsto & \left(g(p)\right)_{p \in \mathcal{Z}(f)(\Fq)}
				\end{array}\right.,\]
				is the composition of the evaluation $\ev_{P_\surj}$ and the projection onto the zero points of $f$. Then, it is surjective and
				\[n_f=\dim \left(\faktor{L(D_\surj)}{\ker \ev_{P_\surj,f}}\right).\] 
				The kernel of $\ev_{P_\surj,f}$, contains the kernel of the map $\ev_{P_\surj}$ and the multiples of $F$, which implies that $n_f \leq \widetilde{n_f}$ and concludes the proof.	
			\end{proof}
			
			We want to handle the quantity $\widetilde{n_f}$ for $f \in L(D) \setminus \ker \ev_P$. Thanks to Th. \ref{DIM}, we are able to pick a monomial basis $\mathcal{B}_{D_\surj}$ of $L(D_\surj)$ modulo $\ker \ev_{P_\surj}$ by reducing lattice points of $P_\surj$ modulo $q-1$ while taking into account the monomial order. We now pick some of these monomials to form a linearly independent family in $L(D_\surj)$ modulo $(\ker \ev_{P_\surj} + F \cdot L(D_\surj - D))$, which provides an easy bound on $n-\widetilde{n_f}$.

			\begin{lemma}\label{CdtDiv}
				For $f \in L(D)$, we denote by $m_f$ the lattice point of $P$ such that $\LM(F)=\chi^\ps{m_f}{P}$. Then
				\[n-\widetilde{n_f}\geq \#\{m \in \Red_<(P_\surj) \st m -m_f \in P_\surj-P\}.\]
			\end{lemma}
			
			\begin{proof}
				
				As the polynomial $f \in L(D)$ is a homogeneous element, the sum of the vanishing ideal $I$ of $\bfX(\Fq)$ and the ideal generated by $f$ in $R$ is also homogeneous. Let $\widehat{\mathcal{G}}$ be a Gr\"obner basis of the ideal $I + \left\langle f \right\rangle$ that contains the Gr\"obner basis $\mathcal{G}$ of $I$ provided by Lem. \ref{NiceBG} and the polynomial $f$. Using Prop. \ref{Grob} and restricting to the component of degree $[D_\surj]$ of $\faktor{R}{I}$, we can deduce that the set
				
				\[\mathcal{M}_f:=\{M \in \M_{D_\surj} \st\forall \: h \in \widehat{\mathcal{G}}, \: \LM(h) \nmid M\}\]
				is a basis of $L(D_\surj)$ modulo $\ker \ev_{P_\surj} + F \cdot L(D_\surj - D)$ of cardinality $\widetilde{n_f}$. Since $\mathcal{G} \subset \widehat{\mathcal{G}}$ and $f \in \widehat{\mathcal{G}}$,
				\[		\mathcal{M}_f \subseteq \{M \in \M_{D_\surj} \st \forall g \in \mathcal{G}, \: \LM(g) \nmid M \AND \LM(f) \nmid M\}.
				\]
				Applying Lem. \ref{NiceBG}, we get
				\[		\mathcal{M}_f \subseteq \{\chi^\ps{m}{P_\surj}, \: m \in \Red_<(P_\surj)  \st \chi^\ps{m_f}{P} \nmid \chi^\ps{m}{P_\surj}\}.
				\]
				
				Since $n = \# \Red_<(P_\surj)$, we have 
				\[n-\tilde{n_f} \geq \# \{m \in \Red_<(P_\surj) \st \chi^\ps{m_f}{P} \mid \chi^\ps{m}{P_\surj}\},\]
				which, thanks to Lem. \ref{CaracDiv}, gives the expected result.
				
				\medskip
			\end{proof}
			
			Combining the two previous lemmas, we are now able to prove Th. \ref{DIST}.
			
			\begin{proof}[Proof of Th. \ref{DIST}]
				By Lem. \ref{DistQuot},
				\[d(\PC_P)\geq \min_{\substack{f \in L(D)\\ f \notin \ker \ev_P}} n - \widetilde{n_f}.\]
				Lem. \ref{CdtDiv} gives a lower bound of $n - \widetilde{n_f}$ in terms of the lattice point $m_f \in P$ such that $\LM(f)=\chi\ps{m_f}{P}$. However, by Lem. \ref{NiceBG}, we can assume that $m_f \in \Red_<(P)$, as the associated monomials in $L(D)$ form a basis of $L(D)$ modulo $\ker \ev_P$. This proves that it is enough to take the minimum over the lattice points of $\Red_<(P)$.
			\end{proof}
			
			For each couple $(<,P')$ of an order on $\Z^N$ and a $P-$surjective polytope, Th. \ref{DIST} provides a lower bound for the minimum distance of $\PC_P$. Computating $\Red_<(P_\surj)$ is easier when $P_\surj$ is not too big but finding a small $P-$surjective polytope may not be trivial, as illustrated in Paragraph \ref{SecToyEx}. It is thus more reasonable to change the ordering. We shall avoid computing several times the reduction of $P$ and $P_\surj$ by computing the ``classes'' modulo $q-1$ once and for all and then taking the minimum elements with respect to different orderings. It may happend that the bound gets sharper when changing the order.

			\section{Examples}		
			
			\subsection{A toy example}\label{SecToyEx}
			
			Let us work out a toy example to illustrate the results on the dimension and the minimum distance in combinatorial terms . Let $P$ be the triangle of vertices $(0,0),(1,0)$ and $(-2,3)$. We want to determine the parameters of the code $\PC_P$ on $\F_4$. The dimension of this code is equal to the number of lattice points of $P$, since $P$ is entirely contained in the square $[0,\dots,q-1]^2$. 
			
			For the minimum distance, we look for a $P-$surjective polygon in the form $P_\surj=\lambda P$  with $\lambda \in \N$. Note that  $\lambda \geq 4$ so that $P_\surj$ has at least $q-1=3$ interior lattice points of each of its edges.
			
			Let us detail the projective reduction modulo $q-1$ of $4P$ (See Fig. \ref{Red4P}). Each edge contains at least 3 points in distinct classes modulo $q-1$ in its interior. Empty circles represent such points that are minimal for the lexicographic order. However, the reduction of the interior of $4P$ contains only 8 classes, whose representatives are depicted by diamonds. Points belonging to the same class are linked by a path in Fig. \ref{Red4P}. A $P-$surjective 	polygon is expected to have $(q-1)^2=9$ classes for its interior lattice points. Therefore, the polygon $4P$ is not $P-$surjective and we need to choose $\lambda$ larger. One can easily see that $\lambda=5$ suits.

			\begin{figure}[ht]
				\centering

				\subcaptionbox{Reduction and classes of $4P$\label{Red4P}}[.4\linewidth]{
					\centering
					\includegraphics{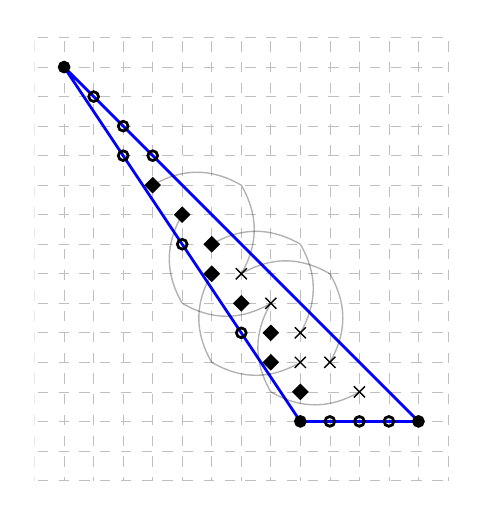}}%
				\subcaptionbox{Reductions of $P$ (black dots) and $5P$ (circles)\label{ToyEx}}[.6\linewidth]{
					\centering
					\includegraphics{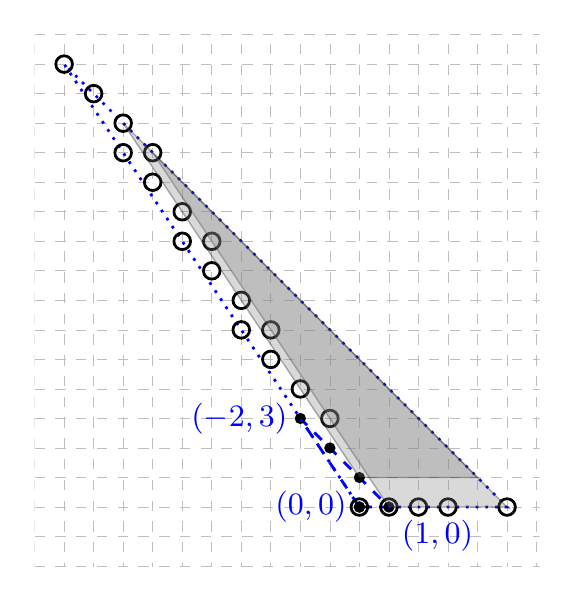}	}
				
				\caption{Reductions of $P$, $4P$ and $5P$ for $P=\Conv((0,0),(1,0),(-2,3))$ on $\F_4$ with respect to the lexicographic order}
			\end{figure}
		
		Now we can bound the minimum distance: make the projective reduction of $P$ and $P_\surj=5P$ with respect to an ordering $<$ on $\Z^2$ \-- here the lexicographic order \-- (see Fig. \ref{ToyEx}) to compute the minimum of Th. \ref{DIST}. 
		It is reached for $m \in \{(1,0), (0,1)\}$: the  cardinality of $(P_\surj - P +m) \cap \Red_<(P_\surj)$ equals 8. A computation with \Magma \: ensures that it is the actual minimum distance.
		
		Choosing $P\surj=\lambda P$ with $\lambda \geq 5$ leads to the same lower bound but the reduction of $P_\surj$ become heavier as $\lambda$ grows.

			\subsection{Towards new champion codes}\label{sec:good-codes}
			
			G. Brown and A. M. Kasprzyk \cite{BK13} systematically investigated toric codes and their generalized versions associated to points in small polygons. This way, they exhibit large families of good codes, acheiving and sometimes beating the best-known parameters \cite{G07}. Given a champion toric code $\C_P$, even if its projective version $\PC_P$ is unlikely to be a champion code itself, it may indicate how to extend $\C_P$ while keeping good parameters.
			
			\begin{figure}[h]
				\centering
				\includegraphics{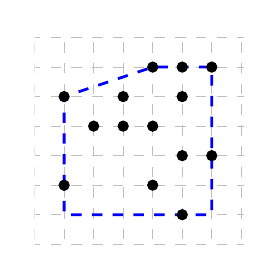}
			\caption{A polygon containing the points defining a champion generalized toric code $[49,14,26]$ over $\F_8$ \cite{BK13}}\label{fig:good-code}
			\end{figure}
			
			\begin{example}\label{ex:good-code}
				Take $U$ the set formed by the black dots in Fig. \ref{fig:good-code}. It defines a champion generalizing toric code $[49,14,26]$ over $\F_8$ \cite{BK13}. Its convex hull does not define a simplicifial toric variety on $\F_8$ since it does not fulfill \ref{H2}. However, it is contained in the the blue dotted polygon $P$ that defines a simplicial toric surface $\bfX_P$ over $\F_8$ and a $[87,14,34]$ code $\PC_P$. With \Magma, evaluating the monomials corresponding to the lattice points in $U$ at torus points but also at two other points gives a $[51,14,27]$ code, and adding again two other points produces a $[53,14,28]$ code. Both of these codes have the best-known parameters \cite{G07}. We have to add three other evaluation points to improve the minimum distance by 1 once more, whereas a $[54,14,29]$ code is already referenced.

			\end{example}
			
			\thanks{The author would like to thank Diego Ruano whose questions and interest motivated the present work, notably the last section.}

			\bibliography{biblio}
			\bibliographystyle{plain}

\end{document}